\documentclass{amsart}

\usepackage{amssymb}
\usepackage{url}

\numberwithin{equation}{section}

\newtheorem{thm}[equation]{Theorem}

\newtheorem{lem}[equation]{Lemma}
\newtheorem{cor}[equation]{Corollary}
\newtheorem{prop}[equation]{Proposition}
\newtheorem{hey}[equation]{Remark}

\newtheorem{eg}[equation]{Example}
\newtheorem{question}[equation]{Question}

\theoremstyle{definition}
\newtheorem{defin}[equation]{Definition}

\def\mod{{\rm mod\ }}
\def\Aut{{\rm Aut}}

\def\Z{{\mathbb Z}}
\def\Cay{{\rm Cay}}
\def\Stab{{\rm Stab}}

\begin{document}

\title{The CI problem for infinite groups}

\author{Joy Morris}
\address{Department of Mathematics and Computer Science \\
University of Lethbridge \\
Lethbridge, AB. T1K 3M4. Canada} \email{joy.morris@uleth.ca}
\thanks{This research was supported in part by the National Science and Engineering Research Council of Canada, Discovery Grant 238552-2011}
\maketitle

\begin{abstract}
A finite group $G$ is a DCI-group if, whenever $S$ and $S'$ are subsets of $G$ with the Cayley graphs Cay$(G,S)$ and Cay$(G,S')$ isomorphic, there exists an automorphism $\varphi$ of $G$ with $\varphi(S)=S'$. It is a CI-group if this condition holds under the restricted assumption that $S=S^{-1}$. We extend these definitions to infinite groups, and make two closely-related definitions: an infinite group is a strongly (D)CI$_f$-group if the same condition holds under the restricted assumption that $S$ is finite; and an infinite group is a (D)CI$_f$-group if the same condition holds whenever $S$ is both finite and generates $G$.

We prove that an infinite (D)CI-group must be a torsion group that is not locally-finite. We find infinite families of groups that are (D)CI$_f$-groups but not strongly (D)CI$_f$-groups, and that are strongly (D)CI$_f$-groups but not (D)CI-groups. We discuss which of these properties are inherited by subgroups. Finally, we completely characterise the locally-finite DCI-graphs on $\mathbb Z^n$. We suggest several open problems related to these ideas, including the question of whether or not any infinite (D)CI-group exists. 
\end{abstract}

\section{Introduction}

Although there has been considerable work done on the Cayley Isomorphism problem for finite groups and graphs, little attention has been paid to its extension to the infinite case.  

\begin{defin}
A Cayley (di)graph $\Gamma=\Cay(G;S)$ is a {\em (D)CI-graph} if whenever $\phi: \Gamma\to \Gamma'$ is an isomorphism, with $\Gamma'=\Cay(G;S')$, there is a group automorphism $\alpha$ of $G$ with $\alpha(S)=S'$ (so that $\alpha$ can be viewed as a graph isomorphism).
\end{defin}

Notice that since $\Aut(\Gamma)=\Aut(\overline{\Gamma})$ (where $\overline{\Gamma}$ denotes the complement of $\Gamma$) and any isomorphism from $\Gamma$ to $\Gamma'$ is also an isomorphism from $\overline{\Gamma}$ to $\overline{\Gamma'}$, a graph is a (D)CI-graph if and only if its complement is also a (D)CI-graph. Since at least one of $\Gamma$ and $\overline{\Gamma}$ must be connected, the problem of determining (D)CI-graphs can be reduced to the connected case.

This definition extends to a definition for groups.

\begin{defin}
A group $G$ is a {\em (D)CI-group} if every Cayley (di)graph on $G$ is a (D)CI-graph.
\end{defin}

These definitions (in the undirected case) as well as the following equivalent condition for a graph to be a (D)CI-graph, first appeared in work by Babai \cite{Baba-1}, extending a research problem posed by \`{A}d\`{a}m for cyclic groups \cite{Adam}.  There has been a large body of work on this topic, and Li published a survey paper \cite{Li-survey} outlining many of the results.   

\begin{thm}\label{CI-alt}\cite{Baba-1}
A Cayley (di)graph $\Gamma$ on the group $G$ is a (D)CI-graph if and only if any two regular copies of $G$ in $\Aut(\Gamma)$ are conjugate.
\end{thm}

In the infinite case, it is natural to consider locally-finite (di)graphs: that is, (di)graphs whose valency is finite.  When studying Cayley (di)graphs, this means that the set $S$ is finite.  However, restricting our consideration to this case complicates matters, as the complement of a locally-finite (di)graph is not locally-finite.  For this reason, the standard argument made above that reduces the finite problem to the case of connected (di)graphs, does not apply to infinite (di)graphs that are locally-finite.  In other words, if one wishes to study this problem in the context of locally-finite (infinite) (di)graphs, it is necessary to consider disconnected as well as connected (di)graphs.

For this reason, we give two new definitions.  In the case of finite (di)graphs, both of these definitions coincide with the definition of a (D)CI-group, but in the infinite case they do not, and are themselves (we believe) worthy of study as natural generalisations of finite (D)CI-groups.

\begin{defin}
A finitely-generated group $G$ is a {\em (D)CI$_f$-group} if every connected locally-finite Cayley (di)graph on $G$ is a (D)CI-graph.
\end{defin}

Note that it is not possible to have a connected locally-finite Cayley (di)graph on a group that is not finitely-generated, so the requirement that the group be finitely-generated only serves to avoid a situation where all non-finitely-generated groups are vacuously CI$_f$-groups.

\begin{defin}
A group $G$ is a {\em strongly (D)CI$_f$-group} if every locally-finite Cayley (di)graph on $G$ is a (D)CI-graph.
\end{defin}

It should be apparent from these definitions that 
$$\text{(D)CI-group} \Rightarrow \text{strongly (D)CI$_f$-group}$$
and if we restrict our attention to finitely-generated groups, 
$$\text{strongly (D)CI$_f$-group} \Rightarrow \text{(D)CI$_f$-group}.$$

In this paper we will construct examples of groups that are (D)CI$_f$-groups but not strongly (D)CI$_f$-groups (despite being finitely generated) and groups that are strongly (D)CI$_f$-groups but not (D)CI-groups, so these definitions are interesting.  We further study these classes, particularly in the case of infinite abelian groups, including a complete characterisation of the locally-finite graphs on $\Z^n$ that are (D)CI-graphs.  We also prove that no infinite abelian group is a (D)CI-group, and that any (D)CI-group must be a torsion group that is not locally finite.  We leave open the question of whether or not any infinite (D)CI-groups exist.

In the only prior work that we are aware of on the CI problem for infinite graphs, Ryabchenko \cite{Ryab} uses the standard definition (the same one we gave above) for a CI-group, and claims to have proven that every finitely-generated free abelian group is a CI-group.  It is clear from his proofs that what he in fact shows is that $\Z$ is a strongly CI$_f$-group, and $\Z^n$ is a CI$_f$-group.  We will restate the results he actually proves in that paper using our terminology, as well as pointing out several consequences of his proofs that he did not mention.  We also show that $\Z^n$ is not a strongly (D)CI$_f$-group if $n>1$.  Ryabchenko cites a paper by Chuesheva as the main motivation for his paper, but the journal is obscure and the url he provides no longer exists, so we were not able to obtain a copy of this paper. L\"{o}h has published a paper \cite{Loeh} on the related question of when a graph can be represented as a Cayley graph on more than one finitely-generated infinite abelian group.

We will proceed from the strongest property to the weakest.  In Section 2, we will consider infinite (D)CI-groups, and prove that various large families of infinite groups cannot be (D)CI-groups; specifically, we show that any infinite CI-group must be a torsion group that is not locally finite.  (Since every DCI-group is a CI-group, this result carries over to the directed case.) In Section 3, we consider strongly CI$_f$-groups.  We construct an infinite family of such groups, but also prove that $\Z^n$ is not a strongly CI$_f$-group for $n>1$.  We show that every finitely-generated subgroup of a strongly CI$_f$-group is a CI$_f$-group, but leave open the question of whether or not all subgroups of strongly CI$_f$-groups are strongly CI$_f$-groups.  In Section 4, we consider CI$_f$-groups.  We show that without the condition of local-finiteness, connectedness is not sufficient to ensure that a Cayley graph on $\Z^n$ is a CI-graph.  We note that $\Z^n$ is a CI$_f$-group for every $n$.  In Section 5, we include the results from \cite{Ryab}.  We have slightly generalised as well as correcting the statements (which can be done using the same proofs), and include some easy corollaries of his proofs, showing that every locally-finite Cayley (di)graph on $\Z^n$ is a normal Cayley (di)graph, and in fact has a unique regular subgroup isomorphic to $\Z^n$. In this pre-print, we include complete proofs of our statements of his results, to avoid making the reader verify that his proofs do what we claim.  Finally, in Section 6, we completely characterise the locally-finite Cayley graphs on $\Z^n$ that are CI-graphs, and in particular show that if the number of connected components of the graph is sufficiently large relative to $n$, then the graph cannot be a CI-graph.

\section{CI-groups}

In this section of the paper, we demonstrate that various families of infinite groups are not CI-groups.  Since all DCI-groups are also CI-groups, this implies that these groups are not DCI-groups.  We also discuss the open questions that remain.

\begin{hey}\label{CI-hereditary}\cite{Baba-Fran-1}
We observe that the property of being a CI-group is inherited by subgroups.  
\end{hey}

There is a standard construction for the above fact, used for finite groups, that works equally well for infinite groups if we are not requiring that graphs be locally finite.  That is: if $H <G$ is not a CI-group, take a connected Cayley graph $\Gamma=\Cay(H;S)$ that is not a CI-graph (use a complement if necessary to ensure that the graph is connected).  Let $\Gamma'=\Cay(H;S')$ be an isomorphic graph that is not isomorphic via an automorphism of $H$.  Then $\Cay(G;S)$ and $\Cay(G;S')$ are clearly isomorphic, but any isomorphism must take connected components to connected components, so would restrict to an isomorphism from $\Gamma$ to $\Gamma'$ that cannot come from a group automorphism of $H$.

We now show that $\Z$ is not a CI-group.  Together with the preceding remark, this has strong consequences.

\begin{prop}\label{Z-not-CI}
The group $\Z$ is not a (D)CI-group.
\end{prop}

\begin{proof}
We prove this by finding a Cayley graph on $\Z$ that is not a CI-graph.  Let $S=\{i \in \Z: i \equiv 1,4 \pmod{5}\}$. We will show that $\Gamma=\Cay(\Z;S)$ is not a CI-graph.

Let $S'=\{i \in \Z: i \equiv 2,3 \pmod{5}\}$, and let $\Gamma=\Cay(\Z;S')$.  We claim that if we define $\phi:\Gamma \to\Gamma'$ by $$\phi(i)=\begin{cases}i \text{ if $i \equiv 0\pmod{5}$}\\ i+1 \text{ if $i \equiv 1 \pmod{5}$}\\ i+2 \text{ if $i \equiv 2 \pmod{5}$}\\ i-2 \text{ if $i \equiv 3 \pmod{5}$}\\  i-1 \text{ if $i \equiv 4\pmod{5}$}\end{cases},$$ then $\phi$ is a graph isomorphism.  Clearly $\phi$ is one-to-one and onto, so we need only show that $xy$ is an edge of $\Gamma$ if and only if $\phi(x)\phi(y)$ is an edge of $\Gamma'$.

Suppose that $xy$ is an edge of $\Gamma$; equivalently, $y-x \equiv 1, 4\pmod{5}$.  A case-by-case analysis of the possible residue classes for $x$ and $y$ shows that this always forces $\phi(y)-\phi(x) \equiv 2,3 \pmod{5}$; equivalently, $\phi(x)\phi(y)$ is an edge of $\Gamma'$.

Since the only automorphisms of $\Z$ fix sets that are closed under taking negatives (which $S$ and $S'$ are), and $S\neq S'$, we conclude that $\Gamma$ is not a CI-graph.
\end{proof}

This of course has very strong consequences.

\begin{cor}\label{non-torsion-not-CI}
No infinite group containing an element of infinite order is a CI-group.  That is, infinite CI-groups must be torsion groups.
\end{cor}

\begin{proof}
If $G$ contains an element $\tau$ of infinite order, then $\langle \tau\rangle \cong \Z$.  By Proposition \ref{Z-not-CI}, this subgroup is not a CI-group, and by Remark \ref{CI-hereditary}, $G$ cannot be a CI-group.
\end{proof}

We now consider infinite abelian $p$-groups.

\begin{prop}\label{p-gps}
No infinite abelian $p$-group is a CI-group.
\end{prop}

\begin{proof}
By Remark \ref{CI-hereditary}, any subgroup of a CI-group is a CI-group.  By Corollary \ref{non-torsion-not-CI}, any infinite CI-group must be a torsion group (i.e., every element has finite order).
Elspas and Turner \cite{Elsp-Turn} showed that $\Z_{16}$ is not a CI-group, and this was generalised in \cite{Baba-Fran-1} to $\Z_{n^2}$ for $n \ge 4$, so any infinite abelian $p$-group would have to be elementary abelian (or contain an infinite elementary abelian subgroup).  But Muzychuk \cite{Muzy-elem-abel} showed that elementary abelian $p$-groups of sufficiently high rank are not CI-groups. (Muzychuk's rank requirement was later improved by Spiga \cite{Spig-CI} and Somlai \cite{Somlai-CI}, but we only require a finite bound.)
\end{proof}

The following simple lemma will allow us to eliminate all infinite abelian groups. This idea has been used in the finite case, but we provide the proof here since it is short, to show that it works equally well in the infinite case.

\begin{lem}\label{iso-subgps}
Suppose that $G$ is a CI-group.  If $H_1, H_2  \le G$ with $|H_1|=|H_2|$ and $|G:H_1|=|G:H_2|$, then some automorphism of $G$ carries $H_1$ to $H_2$.  In particular, $H_1 \cong H_2$.
\end{lem}

\begin{proof}
We have $\Cay(G;H_1-\{e\}) \cong \Cay(G;H_2-\{e\})$ since both consist of $|G:H_1|$ disjoint copies of the complete graph on $|H_1|$ vertices.  So there is an automorphism of $G$ that carries $H_1$ to $H_2$.
\end{proof}

Using the above results, we can now show that no infinite abelian group is a CI-group. In fact the idea of this proof does not really require the assumption that the infinite group is abelian, but that is certainly more than sufficient, and results in the strong corollary that follows.

\begin{thm}\label{abelian-not-CI}
No infinite abelian group is a CI-group.
\end{thm}

\begin{proof}
Suppose that $G$ were an infinite abelian CI-group.
By Corollary \ref{Z-not-CI}, we can assume that every element of $G$ has finite order.  By Proposition \ref{p-gps} (and Remark \ref{CI-hereditary}), we can assume that $G$ does not contain an infinite $p$-group (applying Proposition~\ref{p-gps} requires the assumption that $G$ is abelian).  Thus every $p$-subgroup of $G$ is a finite CI-group, and there are nontrivial $p$-subgroups of $G$ for infinitely many primes.  Fix some prime $p$ for which the $p$-subgroups of $G$ are nontrivial.  Let $H_1$ be any infinite subgroup of $G$ that has infinite order and infinite index in $G$, and has no elements of order $p$. (Such an $H_1$ exists since the Sylow $p$-subgroup of $G$ is finite.  For example, if $P_1, P_2, \ldots$ are all of the nontrivial Sylow subgroups of $G$ with the exception of the Sylow $p$-subgroup, we could take $\langle P_i: i \text{ is odd}\rangle$.)  Let $H_2$ be generated by $H_1$ together with an element of order $p$ from $G$.  Clearly, $H_1$ and $H_2$ are non-isomorphic since only one contains an element of order $p$, but this contradicts Lemma \ref{iso-subgps}.
\end{proof}

A locally-finite group is a group in which every finitely-generated subgroup is finite.  The preceding theorem has the following consequence.

\begin{cor}\label{locally-finite}
No infinite locally-finite group is a CI-group.
\end{cor}

\begin{proof}
Hall and Kulatilaka \cite{Hall-Kula} and Kargapolov \cite{Karg} independently proved that every infinite locally-finite group contains an infinite abelian group.  Both proofs rely on the Feit-Thompson Theorem.  Together with Remark \ref{CI-hereditary}, Theorem \ref{abelian-not-CI} therefore yields the desired conclusion.
\end{proof}

Given the above results, it would be tempting to conjecture that no infinite group is a CI-group, but this is by no means clear, particularly in the case of unusual groups such as the Tarski Monsters (see below).  We leave this as a problem for future research, first summarising what we can say about such a group.

\begin{cor}
Every subgroup of a CI-group must be a CI-group.
Furthermore, every infinite CI-group must be:
\begin{enumerate}
\item a torsion group; and
\item not locally-finite.
\end{enumerate}
In addition, if there is an infinite CI-group, there is one that is finitely generated.
\end{cor}

\begin{proof}
The first statement is Remark \ref{CI-hereditary}.  Conclusion (1) is Corollary \ref{non-torsion-not-CI}.  Conclusion (2) is Corollary \ref{locally-finite}.  

Suppose now that $G$ is an infinite CI-group.  Since $G$ is not locally-finite, it must have a subgroup that is finitely generated but infinite, and is still a CI-group (by Remark \ref{CI-hereditary}).  
\end{proof}

In determining whether or not there is an infinite CI-group, one possible family of candidates that needs to be considered carefully is the family of so-called ``Tarski Monsters".  These are infinite groups whose only proper subgroups have order $p$ for some fixed (but dependent upon the group) large prime $p$.  Thus, every element of the group has order $p$, while any two elements in different cyclic subgroups generate the entire group.  Clearly, if such a group were to be a CI-group, then every non-identity element would have to lie in a single automorphism class (otherwise, if there is no automorphism taking $g$ to $h$ in the Tarski monster $G$, then $\Cay(G;\{g, g^{-1}\}) \cong \Cay(G;\{h, h^{-1}\})$ but there is no automorphism of $G$ taking $\{g, g^{-1}\}$ to $\{h, h^{-1}\}$).  We found discussions on the internet \cite{discussion} indicating that for some Tarski monsters, any two of the subgroups are conjugate, but did not find an answer as to whether or not the stronger condition we are interested in is true for some Tarski monsters.  Even if it were true, this is not enough to guarantee that such a group is a CI-group.  We leave this as an open question.

\begin{question}
Does there exist an infinite CI-group?  In particular, is any Tarski monster a CI-group?
\end{question}

\section{Strongly CI$_f$-groups}

In contrast to the class of CI-groups, we were able to find groups that are strongly CI$_f$-groups. To begin this section, we note that Ryabchenko \cite{Ryab} proved that $\Z$ is a strongly CI$_f$-group.  This result is stated in Section \ref{sect-Ry} of this paper, as Corollary \ref{Z-is-str-CI_f}.  

This naturally leads to the question of $\Z^n$.  We show that $\Z^n$ is not a strongly CI$_f$-group for any $n>1$.  Because we actually plan to give a precise characterisation of the finitely-generated (D)CI-graphs on $\Z^n$, we in fact prove a stronger result.

\begin{prop}\label{non-CI-on-Zn}
Let $n>1$, and let $\Gamma=\Cay(\Z^n;S)$ be any Cayley (di)graph on $\Z^n$ such that the number of connected components of (the underlying graph of) $\Gamma$ is either infinite, or is divisible by $p^2$ for some prime $p$.  Then $\Gamma$ is not a (D)CI-graph.
\end{prop}

\begin{proof}
For this proof, we use the formulation of the CI problem given in Theorem \ref{CI-alt}.

Let $G=\langle S \rangle$, and let $\Gamma_0=\Cay(G;S)$ (so this is connected).  Then $\Aut(\Gamma)$ will either be $S_{\Z} \wr \Aut(\Gamma_0)$, or $S_n \wr \Aut(\Gamma_0)$, where $n$ is finite and there is some prime $p$ such that $p^2 \mid n$.  Consider the subgroup of the appropriate symmetric group that is induced by the natural action of $\Z^n$ on the connected components of $\Gamma$.  Clearly this will be a regular abelian subgroup that is either countably infinite, or of order $n$. There are many nonisomorphic countably infinite regular abelian subgroups of $S_{\Z}$ ($\Z$ and $\Z_2 \times \Z$, for example).  Likewise, there are at least two nonisomorphic regular subgroups of $S_n$ ($\Z_p \times \Z_{n/p}$ and $\Z_{p^2} \times \Z_{n/p^2}$).  Since $n>1$, each of these can be expanded to a regular action isomorphic to $\Z^n$ in $\Aut(\Gamma)$.  Since the subgroups are nonisomorphic, they are not conjugate in the appropriate symmetric group, so the expanded actions on $\Gamma$ are not conjugate in $\Aut(\Gamma)$.  Thus $\Gamma$ is not a (D)CI-graph.
\end{proof}

\begin{cor}\label{Zn-not-str-CI}
The group $\Z^n$ is not a strongly CI$_f$-group for $n >1$.
\end{cor}

\begin{proof}
When $n>1$, it is easy to construct finitely-generated Cayley graphs on $\Z^n$ for which the number of connected components is either countably infinite, or divisible by a square.  For example, $\Gamma_1=\Cay(\Z^n;\{\pm (1,0,\ldots, 0)\})$ has a countably infinite number of connected components, while $\Gamma_2$, the Cayley graph on $\Z^n$ whose connection set is the standard generating set for $\Z^n$ (together with inverses) with the first generator (and its inverse) replaced by $\pm (p^2,0, \ldots, 0)$, will have $p^2$ connected components.  So Proposition \ref{non-CI-on-Zn} is sufficient.

Had we only wanted to show that $\Z^n$ is not a strongly CI$_f$-group for $n>1$, we could have pointed out that $\Gamma_1 \cong \Cay(\Z^n;\{\pm(2,0,\ldots,0)\})$ but not via a group automorphism of $\Z^n$, or similarly that $\Gamma_2$ is isomorphic to the Cayley graph on $\Z^n$ whose connection set is the standard generating set for $\Z^n$ (together with inverses) with the first generator (and its inverse) replaced by $\pm(p,0,\ldots,0)$, and the second generator (and its inverse) replaced by $\pm (0,p,\ldots, 0)$, but not via a group automorphism of $\Z^n$.
\end{proof}

Having determined the status of free abelian groups, we turn our attention to the opposite end of the spectrum of infinite abelian groups and consider torsion groups.  First we prove a restriction on torsion groups that are strongly CI$_f$-groups (dropping the abelian constraint for the time being).

\begin{lem}\label{torsion-str-cont-finite-CI}
Suppose that $G$ is a locally-finite torsion group that is a strongly CI$_f$-group.  Then
every finite subgroup of $G$ is a CI-group.  

Furthermore, for $p \ge 5$ the Sylow $p$-subgroups of $G$ are elementary abelian, and the Sylow $3$-subgroups are either cyclic of order at most 27, or elementary abelian.
\end{lem}

\begin{proof}
Since $G$ is a strongly CI$_f$-group, an argument similar to that of Remark \ref{CI-hereditary} shows that every finite subgroup must be a CI-group.

Babai and Frankl \cite{Baba-Fran-1} showed that for $p \ge 5$ the only finite $p$-groups that are CI-groups are elementary abelian, and the finite $3$-groups that are CI-groups are either cyclic of order at most 27, or elementary abelian.  Furthermore, Muzychuk \cite{Muzy-elem-abel} proved that elementary abelian groups of sufficiently high rank are not CI-groups.  Since $G$ is locally-finite and the results just stated imply that every finite $p$-subgroup has bounded order, there must be a finite number of generators that contribute to any $p$-group in $G$.  In particular, this means that the $p$-groups in $G$ must all be finite.  Thus by \cite{Baba-Fran-1} again, we obtain the desired conclusion.
\end{proof}

In addition to the single example of $\Z$, we are able to find an infinite family of groups are strongly CI$_f$-groups.  

\begin{thm}\label{abelian-torsion-str}
Let $G$ be a countable abelian torsion group.  Then $G$ is a strongly (D)CI$_f$-group if and only if every finite subgroup of $G$ is a (D)CI-group.
\end{thm}

\begin{proof}
Abelian torsion groups are locally-finite, so
necessity is shown in Lemma \ref{torsion-str-cont-finite-CI}.

For the converse, suppose that $G$ is a countable abelian torsion group, and every finite subgroup of $G$ is a (D)CI-group.  

By Lemma \ref{torsion-str-cont-finite-CI}, the Sylow $p$-subgroups of $G$ are elementary abelian, or cyclic of order at most 27, where $p \ge 3$.
Aside from some finite exceptional groups whose order does not exceed $2^5 3^2=288$, it is known that in any finite abelian (D)CI-group $H$, every Sylow $p$-subgroup of $H$ must be either $\Z_4$, or elementary abelian.  This strengthening of the work of Babai and Frankl \cite{Baba-Fran-1} for $p=2$ and $p=3$ is mentioned in \cite{Li-Lu-Palf}.  Since $G$ has arbitrarily large finite subgroups all of which are (D)CI-groups, this implies that every Sylow $p$-subgroup of $G$ must be either $\Z_4$, or elementary abelian.

Let $\Gamma=\Cay(G;S)\cong \Gamma'=\Cay(G;S')$, with $S$ finite.  Since $G$ is an abelian torsion group, $\langle S \rangle$ must be finite, and $\langle S' \rangle$ has the same finite order, so $H=\langle S, S' \rangle$ is a finite subgroup of $G$, so is a (D)CI-group.  Clearly $\Cay(H;S) \cong \Cay(H;S')$, so as $H$ is a (D)CI-group, there is an automorphism $\alpha$ of $H$ taking $S$ to $S'$.  

Since $G$ is countable, list the elements of $G$: $g_1, g_2, \ldots$, so that $H=\{g_1, \ldots, g_{|H|}\}$ (the rest of the list can be arbitrary).  For $i \ge |H|$, define $G_i = \langle g_1, \ldots, g_i\rangle$ (so $G_{|H|}=H$). 

We claim that for $i \ge |H|$, there is an automorphism $\alpha_i$ of $G_i$ that takes $S$ to $S'$ (so is an isomorphism from $\Cay(G_i;S)$ to $\Cay(G_i;S')$) such that for every $j \in \{|H|, |H|+1, \ldots, i\}$, the restriction of $\alpha_i$ to $G_j$ is $\alpha_j$.  We prove this claim by induction.  The base case of $i=|H|$ has been established.  By induction, we can assume that we have $\alpha_{i-1}$ such that the restriction of $\alpha_{i-1}$ to $G_j$ is $\alpha_j$ for every $|H|\le j \le i-1$, so we need only find $\alpha_i$ such that the restriction of $\alpha_i$ to $G_{i-1}$ is $\alpha_{i-1}$.  Since $G_i$ is abelian, it is the direct product of its Sylow $p$-subgroups, so if we show that the action of $\alpha_{i-1}$ on any Sylow $p$-subgroup of $G_{i-1}$ is the restriction of the action of $\alpha_i$ on the corresponding Sylow $p$-subgroup of $G_i$, this will suffice.  Let $P_i$ be a Sylow $p$-subgroup of $G_i$, and $P_{i-1}$ the corresponding Sylow $p$-subgroup of $G_{i-1}$.  If $P_{i-1}=P_i$ then we define $\alpha_i(g)=\alpha_{i-1}(g)$ for every $g \in P_i=P_{i-1}$.  If $P_i$ is elementary abelian and $P_i \neq P_{i-1}$, then since $G_i=\langle G_{i-1},g_i\rangle$ is abelian, we must have $P_i\cong P_{i-1} \times \Z_p$.  In this case use this representation, and for any $(g,h) \in P_i=P_{i-1} \times\Z_p$, define $\alpha_i(g,h)=(\alpha_{i-1}(g),h)$.  The only remaining possibility is that $p=2$, $P_i=\Z_4$,  and $P_{i-1}=\Z_2$.  In this case, define $\alpha_i(g)=g$ for every $g \in P_i$.  Since $\alpha_{i-1}$ must act as the identity on $P_{i-1}\cong \Z_2$, the restriction of $\alpha_i$ to $P_{i-1}$ is again $\alpha_{i-1}$.

Now we define $\alpha'$, which will be an automorphism of $G$ that takes $S$ to $S'$. For ease of notation, first define $\alpha_i =\alpha$ for $1 \le i \le |H|$.  Now for any $g_i \in G$, define $\alpha'(g_i)=\alpha_i(g_i)$.  We show that the map $\alpha'$ is an automorphism of $G$. Let $g_i, g_j \in G$ with $i\le j$. First, notice that because the restriction of $\alpha_j$ to $G_i$ is $\alpha_i$ (where $G_i=H$ for every $1 \le i \le |H|$), we have $\alpha_j(g_i)=\alpha_i(g_i)$. Now, $g_i, g_j, g_ig_j \in G_j$ and $$\alpha'(g_i)\alpha'(g_j)=\alpha_i(g_i)\alpha_j(g_j)=\alpha_j(g_i)\alpha_j(g_j)=\alpha_j(g_ig_j)=\alpha'(g_ig_j).$$ 
\end{proof}

While the finite abelian (D)CI-groups have not been completely determined, elementary abelian groups of rank at most 4 are known to be DCI-groups \cite{Ted-Z_p3,Gods-Z_p2,Hira-Muzy,thesis,Turn-Zp}.  So the preceding theorem gives us an infinite class of infinite strongly (D)CI$_f$-groups: namely, pick any infinite set of primes $Q$.  For each $p \in Q$, take a cyclic $p$-group.  Define $G$ to be the direct product of the chosen groups.  Then $G$ is a strongly (D)CI$_f$-group. (It would be nice to be able to select an elementary abelian $p$-group of rank higher than one for at least some of the primes in $Q$; unfortunately, the question of whether or not finite direct products of most such groups are (D)CI-groups remains open.)

It is, unfortunately, not clear whether the property of being a strongly (D)CI$_f$-group is necessarily inherited by subgroups of strongly (D)CI$_f$-groups.  In the examples that we have found, it is inherited, since the only infinite subgroup of $\Z$ is $\Z$, and if $G$ is any group in the family of strongly (D)CI$_f$-groups described in Theorem \ref{abelian-torsion-str}, and $H$ is any infinite subgroup of $G$, then (by our structural characterisation of the family) $H$ is in the family, so $H$ is a strongly (D)CI$_f$-group.  In general, though, we do not see why the following situation might not arise:  $G$ is a strongly (D)CI$_f$-group, and for some infinite subgroup $H$ and some finite subsets $S, S'$ of $G$, $\Cay(G; S) \cong \Cay(G;S')$, but for every automorphism $\alpha$ of $G$ that takes $S$ to $S'$, we have $\alpha(H) \neq H$, and in fact no automorphism of $H$ takes $S$ to $S'$.  

\begin{question}
Is every subgroup of a strongly (D)CI$_f$-group necessarily a strongly (D)CI$_f$-group?
\end{question}

We can at least say that subgroups of strongly (D)CI$_f$-groups that are finitely-generated are necessarily (D)CI$_f$-groups.  

\begin{prop}\label{str-CI_f-to-CI_f}
Every finitely-generated subgroup of a strongly (D)CI$_f$-group is a (D)CI$_f$-group.
\end{prop}

\begin{proof}
Let $G$ be a strongly (D)CI$_f$-group, and let $H\le G$ be finitely generated.  Suppose that $\langle S \rangle=H$, and $\Cay(H;S) \cong \Cay(H;S')$ for some subset $S'$ of $H$.  Since $\Cay(H;S)$ (or the underlying undirected graph) is connected, we also have $\langle S' \rangle =H$.  Clearly, $\Cay(G;S) \cong \Cay(G;S')$ since each is the disjoint union of $|G:H|$ copies of the original (di)graph.  Since $G$ is a strongly (D)CI$_f$-group, there is an automorphism $\alpha$ of $G$ such that $\alpha(S)=S'$.  Since $H=\langle S \rangle=\langle S' \rangle$, we must have $\alpha(H)=H$, so the restriction of $\alpha$ to $H$ is an automorphism of $H$ that takes $S$ to $S'$.
\end{proof}

\section{CI$_f$-groups}

Although it was not the statement he gave, Ryabchenko \cite{Ryab} proved that $\Z^n$ is a CI$_f$-group for every $n$; that is, every finitely-generated free abelian group is a CI$_f$-group.  We include a slight generalisation of his proof in Section \ref{sect-Ry}, as Corollary \ref{Zn-is-CI_f}.  Currently, these are the only infinite (D)CI$_f$-groups that we know of, since the family of strongly (D)CI$_f$-groups determined in Theorem \ref{abelian-torsion-str} has no finitely-generated members.

An interesting observation is that although connectedness is enough to ensure that a locally-finite Cayley graph on $\Z^n$ is a (D)CI-graph, it is not sufficient if the graph is not locally-finite.

\begin{cor}\label{locally-infinite}
Let $n>1$. Amongst connected Cayley (di)graphs on $\Z^n$ that are not locally finite, some will be (D)CI-graphs and some will not.
\end{cor}

\begin{proof}
Corollary~\ref{Zn-is-CI_f} tells us that any such (di)graph for which the complement is locally finite and connected will be (D)CI, while Proposition \ref{non-CI-on-Zn} tells us that any such (di)graph for which the complement is locally finite with a number of connected components that is infinite or is not square-free, will not be (D)CI.
\end{proof}

Since subgroups of finitely-generated groups need not be finitely-generated, it is again not at all evident whether or not the property of being a (D)CI$_f$-group is inherited by subgroups.  Amongst other things, we would need to determine that all subgroups of (D)CI$_f$-groups are finitely generated.  Setting this aside, it is not evident whether or not finitely generated subgroups of (D)CI$_f$-groups are (D)CI$_f$-groups.  Since for a (D)CI$_f$-group we only know that connected, locally-finite Cayley (di)graphs are (D)CI-graphs, it is hard to see even how, given two locally-finite, isomorphic Cayley (di)graphs on $H \le G$, one might construct suitable Cayley (di)graphs on $G$ that are locally-finite and connected, to use the (D)CI$_f$-property.  One possible approach would involve proving that every Cayley colour graph on $G$ actually has the CI-property, and then using a second colour of edges on a finite number of generators to connect cosets of $H$. We leave this as another question.  To prove any result along these lines (e.g. with the additional condition that $|G:H|$ be finite) would be interesting, we believe.

\begin{question}
If $G$ is a (D)CI$_f$-group and $H \le G$ is finitely-generated, is $H$ a (D)CI$_f$-group?
\end{question}

\section{Ryabchenko's results}\label{sect-Ry}

In this section we state the results from Ryabchenko's paper, and some closely-related results.

Although Ryabchenko does not consider digraphs, his proofs in fact cover the more general situation, and have a number of easy and interesting consequences that he does not make note of.

\begin{thm}[\cite{Ryab}, Theorem 1]\label{Ry-Z}
Let $S \subset \Z$ be finite.  If $\Cay(\Z;S') \cong \Cay(\Z;S)$ then $S'=\pm S$.
\end{thm}

\begin{proof}
Let $S \subset \Z$ be finite.  Arrange the elements of $S$ in order of non-decreasing magnitude, so  $S=\{s_1, \ldots,  s_k\}$ with $| s_i| \le | s_{i+1}|$ for every $1\le i \le k-1$, with equality if and only if $ s_i=- s_{i+1}$.
Let $\Gamma=\Cay(\Z:S)$ and suppose that $\Gamma'=\Cay(\Z;S') \cong \Gamma$ via the isomorphism $\phi$.  We will show that $S'=\pm S$.

We prove this by induction on $k$.  For the base case of $k=0$, $\Gamma$ and $\Gamma'$ are both empty, so $S=S'=\emptyset$.

Colour the directed edges formed by $s_k$ in $\Gamma$ red; this forms a red 2-factor in $\Gamma$.  The image of this red 2-factor under $\phi$ must be a red 2-factor in $\Gamma'$.  Consider vertices $x'$, $y'=x'+s'_i$ and $z'=y'+s'_j$ in $\Gamma'$ that are consecutive in the red 2-factor, so $x'y'$ and $y'z'$ are red edges in $\Gamma'$.  Let $x$ and $z$ be $\phi^{-1}(x')$ and $\phi^{-1}(z')$ (respectively).  By the maximality of $|s_k|$ in $S$, the red path of length 2 is the unique path of length 2 from $x$ to $z$ in $\Gamma$.  Hence the red path of length 2 must be the unique path of length 2 from $x'$ to $z'$ in $\Gamma'$.  Since $\Z$ is abelian, if $s'_j\neq s'_i$, then the path from $x'$ to $x'+ s'_j$ to $x'+s'_j+s'_i=z'$ would be a different path of length 2 from $x'$ to $z'$, a contradiction.  Hence in any connected component of the red subgraph of $\Gamma'$, every edge comes from some fixed generator $s'_i$ of $\Gamma'$ (or from the pair $s'_i$ and $-s'_i$).

Take an arbitrary connected component of the red 2-factor in $\Gamma'$, and let the corresponding generator of $\Gamma'$ be $s'_i$, and $x'$ a vertex of this component.
We now show that $|s'_i|$ is maximal in $S'$.  If it were not, suppose that for some $j$, $|s'_j|>|s'_i|$.  Then $|s'_i|s'_j$ gives a path of length $|s'_i|$ in one direction or the other between $x'$ and $x'+|s'_j|s_i$.  As before, the pre-images of these vertices under $\phi$ have a red path of length $|s'_j|$ between them (possibly one in each direction), and by the maximality of $s_k$, this is the shortest path.  But $|s'_i|<|s'_j|$, so the pre-image of the path we have just found is shorter, a contradiction that proves our claim.  Since our choice of the connected red component was arbitrary, this in fact shows that every red edge comes from a maximal element of $S'$. Since $|S|=|S'|$, we may assume that the red edges all come from $s'_k$.

Next we show that $s'_k=s_k$.  Since $\phi$ maps the red 2-factor of $\Gamma$ to the red 2-factor of $\Gamma'$, both must have the same number of connected components.  But it is easy to see that the red 2-factor of $\Gamma$ has $|s_k|$ connected components, and the red 2-factor of $\Gamma'$ has $|s'_k|$ connected components, so these are equal.

By our inductive hypothesis and replacing $S'$ with $-S'$ if necessary, we can assume that $S-\{s_k\}=S'-\{s'_k\}$, so the only possible problem arises if $s_k=-s'_k$ and $S-\{s_k\}$ contains $s_i$ but not $-s_i$ for some $i$.  Choose the largest $i$ for which this holds.  By the maximality of $i$ and the ordering of $S$, there is a unique path of length $|s_k|$ in $\Gamma$ that goes from $0$ to $|s_k|s_i$.  We will colour this path green.  Again by the maximal length of $s_i$ under the given condition, $\phi$ must take this to a unique (now green) path of length $|s_k|$ in $\Gamma'$ that goes from some vertex $x'$ to $x'+|s_k|s_i$.
The unique (red) path of length $|s_i|$ between $0$ and $|s_k|s_i$ either travels in the same direction as the green path, or in the opposite direction, and this is determined by the sign of $s_k$.  If the red and green paths travel in the same direction in $\Gamma$, they must also travel in the same direction in $\Gamma'$, and conversely, so the sign of $s'_k$ must be the same as the sign of $s_k$.
\end{proof}

This has the following immediate consequence.

\begin{cor}\label{Z-is-str-CI_f}
The group $\Z$ is a strongly $(D)CI_f$-group.
\end{cor}

\begin{proof}
If $\Cay(\Z;S)$ and $\Cay(\Z;S')$ are isomorphic and $S$ is finite, then by Theorem \ref{Ry-Z}, $S'=\pm S$, so either the identity or the automorphism of $\Z$ that takes every integer to its negative will act as an isomorphism from $\Cay(\Z;S)$ to $\Cay(\Z;S')$.
\end{proof}

The next result does not look at all like the statement of Theorem 2 from~\cite{Ryab}, but is the clearest and most precise statement of the proof he gives for that theorem.

\begin{thm}[\cite{Ryab}, Theorem 2]\label{isos-on-Zn-are-autos}
Let $S$ be a finite generating set for $\Z^n$, and let $\Gamma=\Cay(\Z^n;S)$.  Then if $\Gamma'=\Cay(\Z^n;S')$ and there is an isomorphism $\phi:\Gamma\to\Gamma'$ such that $\phi$ takes the identity of $\Z^n$ to the identity of $\Z^n$, then $\phi$ is a group automorphism of $\Z^n$.
\end{thm}

\begin{proof}
The elements of $S$ can be thought of as vectors in $\mathbb R^n$, and their magnitude calculated under the usual metric for this vector space. As in the proof of Theorem \ref{Ry-Z}, order the elements of $S$ in non-decreasing order of magnitude under this measure.  Let $\Gamma$, $\Gamma'$, and $\phi$ be as in the statement of the theorem.

{\bf Claim 1.} If $s \in S$ has the property that for any vertex $x$ of $\Gamma$, we have $\phi(x+s)-\phi(x)=\phi(x+2s)-\phi(x+s)$ (that is, any two consecutive edges of $\Gamma'$ that are the images of edges of $\Gamma$ that come from $s$, themselves come from the same $s' \in S'$), then in fact there is a fixed element $s' \in S'$ such that for every vertex $x$ of $\Gamma$, we have $\phi(x+s)-\phi(x)=s'$ (that is, any edge of $\Gamma'$ that is the image of an edge of $\Gamma$ that comes from $s$, itself comes from $s'$).  Essentially, this says that if each ray in $\Gamma$ formed by $s$ maps to a ray in $\Gamma'$ formed by some element of $S'$, then all of these rays in $\Gamma'$ must actually be formed by the same fixed element of $S'$.

To prove Claim 1, assume that $s$ is a counterexample to Claim 1.  Thus, $s$ satisfies the hypothesis of the claim, but some two rays of $\Gamma$ formed by $s$ are mapped to rays of $\Gamma'$ that are formed by two distinct elements $s'$ and $s''$ of $S'$.  Since $\Gamma$ is connected, there is a path in $\Gamma$ from a vertex of the first ray to a vertex of the second ray, and somewhere along this path there is a first ray of $\Gamma$ that maps to a ray of $\Gamma'$ formed by $s''$.  Thus, without loss of generality, we may assume that there is a vertex $x$ of $\Gamma$ and a single element $s_i \in S$, such that $\phi(x+s)-\phi(x)=s'$, but $\phi(x+s_i+s)-\phi(x+s_i)=s''$.  Now since $\Z^n$ is abelian, for any $j$ we have that $x+js$ is adjacent to $x+s_i+js$ via an edge that comes from $s_i$.  Hence $\phi(x+js)$ must be adjacent to $\phi(x+s_i+js)$, and by assumption, $\phi(x+js)=\phi(x)+js'$, while $\phi(x+s_i+js)=\phi(x+x_i)+js''$.  Since $j$ is running through $\Z$ but $S'$ is finite, by the pigeon-hole principle sooner or later some pair of such adjacencies must come from the same generator $s'_k \in S'$.  In other words, there exist distinct $j_1$ and $j_2$ such that $\phi(x+j_2s)+s_k'=\phi(x+j_1s)+s_k'+(j_2-j_1)s''$.  But this means that $\phi(x)+j_2s'=\phi(x)+j_1s'+(j_2-j_1)s''$, so $(j_2-j_1)s'=(j_2-j_1)s''$.  Since $j_2 \neq j_1$, this forces $s''=s'$, a contradiction that proves Claim 1.

{\bf Claim 2.} The hypothesis of  Claim 1 actually does hold for any vertex $x$ of $\Gamma$ and any $s \in S$.  That is, we show that $\phi(x+s)-\phi(x)=\phi(x+2s)-\phi(x+s)$. In other words, any two consecutive edges of $\Gamma'$ that are the images of edges that come from a fixed element $s \in S$ of $\Gamma$, themselves come from the same $s' \in S'$.  

Towards a contradiction to Claim 2, let $i$ be maximized subject to the condition that $s_i$ violates this property.  So there is some vertex $x$ of $\Gamma$ such that $\phi(x+s_i)-\phi(x)=s'_j$ and $\phi(x+2s_i)-\phi(x+s_i)=s'_k$, and $k \neq j$.  Now there is a second path of length 2 in $\Gamma'$ from $\phi(x)$ to $\phi(x+2s_i)$, that travels via the vertex $\phi(x)+s'_k \neq \phi(x)+s'_j=\phi(x+s_i)$.  Hence there is a second path of length 2 in $\Gamma$ that goes from $x$ to $x+2s_i$, via the vertex $y=\phi^{-1}(\phi(x)+s'_k)$.  Since $\phi(x+s_i)=\phi(x)+s'_j$, we have $y \neq x+s_i$, so $y-x$ and $x+2s_i-y$ are distinct elements of $S$.  By the triangle inequality, at least one of $y-x$ and $x+2s_i-y$ must be longer than $s_i$.  Since $y-x, x+2s_i-y \in S$, let $s_{\ell}$ be one of these vectors that is longer than $s_i$.  By our ordering of $S$, we have $\ell>i$.  So by our choice of $i$, we must have $s_\ell$ satisfies the hypothesis of Claim 1, so that there is some $s' \in S'$ such that every edge of $\Gamma$ that comes from $s_\ell$ maps to an edge of $\Gamma'$ that comes from $s'$.  Thus, either $s'_j=s'$ or $s'_k=s'$.  But we have $\phi(x+s_i)-\phi(x)=s'_j$ and $\phi(x+2s_i)-\phi(x+s_i)=s'_k$, a contradiction that proves Claim 2.

Since $\phi$ maps the identity of $\Z^n$ to the identity of $\Z^n$, we clearly have $\phi(S)=S'$ since these are the neighbours of the identity.  We can therefore list the elements of $S'$ as $\phi(s_1), \ldots, \phi(s_k)$, where $S=\{s_1, \ldots, s_k\}$.  Now Claims 1 and 2 together establish that for any vertex $x$ of $\Gamma$, if $x=a_1s_1+\ldots +a_ks_k$, then $\phi(x)=a_1\phi(s_1) +\ldots+a_k\phi(s_k)$.  Since $\Gamma$ is connected, $\langle S \rangle =\langle S' \rangle =\Z^n$, so this shows that $\phi$ is in fact an automorphism of $\Z^n$.
\end{proof}

This has an easy corollary, which is (except for his omission of his assumption that the graphs are locally-finite) the result that was stated in \cite{Ryab}, Theorem 2.

\begin{cor}\label{Zn-is-CI_f}
The group $\Z^n$ is a $(D)CI_f$-group.
\end{cor}

\begin{proof}
Let $\Gamma=\Cay(\Z^n;S)$ and $\Gamma'=\Cay(\Z^n;S')$ with $\phi:\Gamma\to\Gamma'$ an automorphism.  Let $\bf 0$ represent the identity element of $\Z^n$.  If $c$ is the element of $\Z^n$ that corresponds to the vertex $\phi(\bf 0)$, then $\phi'=\phi-c$ is an isomorphism from $\Gamma$ to $\Gamma'$ that takes $\bf 0$ to $\bf 0$.  By Theorem \ref{isos-on-Zn-are-autos}, $\phi'$ must be an automorphism of $\Z^n$.
\end{proof}

The following corollary was not mentioned in Ryabchenko's paper but is an immediate consequence of his proof.

\begin{cor}\label{normal}
If $\Gamma=\Cay(\Z^n;S)$ for some finite generating set $S$ of $\Z^n$, then $\Gamma$ is a normal Cayley (di)graph of $\Z^n$.
\end{cor}

\begin{proof}
Let ${\bf 0}$ be the vertex of $\Gamma$ corresponding to the identity element of $\Z^n$.
Let $\gamma$ be any automorphism of $\Gamma$.  Then $\gamma$ fixes ${\bf 0}$.  By Theorem \ref{isos-on-Zn-are-autos}, $\gamma \in \Aut(\Z^n)$, so $\Z^n \triangleleft \Aut(\Gamma)$.
\end{proof}

The final corollary presented in this section is slightly less obvious, but is still essentially a consequence of the proof in \cite{Ryab}.

\begin{cor}\label{unique}
If $\Gamma=\Cay(\Z^n;S)$ for some finite generating set $S$ of $\Z^n$, then $\Aut(\Gamma)$ has a unique regular subgroup isomorphic to $\Z^n$.
\end{cor}

\begin{proof}
Let $Z_1$ and $Z_2$ be two regular subgroups isomorphic to $\Z^n$ in $\Aut(\Gamma)$ (with $Z_1=\langle S \rangle$).  Let $\alpha' \in Z_2$ be arbitrary; we plan to show that $\alpha' \in Z_1$.  Let $\alpha \in Z_1$ such that $\alpha'(\bf 0)=\alpha(\bf 0)$, where $\bf 0$ is the vertex corresponding to the identity of $\Z^n$.  Then $\beta=\alpha^{-1}\alpha' $ is an automorphism of $\Gamma$ that fixes $\bf 0$, so by Theorem \ref{isos-on-Zn-are-autos}, $\beta \in \Aut(Z_1)$.

Since $S$ is finite, $Z_1$ and $Z_2$ each have finite index in $\Aut(\Gamma)$.  It is well-known that the intersection of two groups of finite index, itself has finite index (c.f. Problem 6, Section 5.1, \cite{Ash}).  Let $Z=Z_1\cap Z_2$.  Clearly, since $Z_1$ and $Z_2$ are abelian, $\beta$ commutes with every element of $Z$.  But since $\beta \in \Aut(Z_1)$, it can only commute with the elements of $Z$ if it fixes all of them.  This means that $\beta$ fixes a finite-index subgroup of $Z_1$ pointwise, so since $\beta \in \Aut(Z_1)$, we must have $\beta=1$.  Hence $\alpha'=\alpha \in Z_1$, as claimed.  Since $\alpha'$ was arbitrary, $Z_2=Z_1$ is the unique regular subgroup isomorphic to $\Z^n$ in $\Aut(\Gamma)$.
\end{proof}

\section{Characterisation of locally-finite (D)CI-graphs on $\Z^n$}

We have already seen that $\Z$ is a strongly DCI$_f$-group, and that for $n>1$, $\Z^n$ is a DCI$_f$-group but not a strongly (D)CI$_f$-group.  The goal of this section of the paper is to give a precise characterisation of the locally-finite Cayley (di)graphs on $\Z^n$ that are (D)CI-graphs (where $n>1$).  Throughout the remainder of this section, we assume $n>1$.

We have also seen that if the number of connected components of a locally-finite Cayley (di)graph on $\Z^n$ is either infinite or divisible by a square, then the graph is not a (D)CI-graph.

To prove our characterisation, we will need the following well-known corollary of Smith normal form (cf.~4.6.1 of~\cite{Ash}).

\begin{thm}[Simultaneous Basis Theorem]\label{simult-basis}
Let $M$ be a free abelian group of finite rank $n \ge 1$ over $\Z$, and let $H$ be a subgroup of $M$ of rank $r$.  Then there is a basis $\{y_1, \ldots, y_n\}$ for $M$ and nonzero elements $a_1, \ldots, a_r \in \Z$ such that $r \le n$, $a_i$ divides $a_{i+1}$ for all $i$, and $\{a_1y_1, \ldots, a_ry_r\}$ is a basis for $H$.
\end{thm}

\begin{cor}\label{conj-H}
Let $H = b_1\Z \times \ldots \times b_n\Z$ 
for some $b_1, \ldots, b_n$ with 
$\Pi_{i=1}^n b_i=k$, where $k$ is finite and square-free.  Then there is an automorphism $\sigma$ of $\Z^n$ such that $H^\sigma= k\Z \times \Z^{n-1}$.
\end{cor}

\begin{proof}
By Theorem~\ref{simult-basis}, there is a basis $\{y_1, \ldots, y_n\}$ for $\Z^n$ and nonzero integers $a_1, \ldots a_n$ such that $a_i$ divides $a_{i+1}$ for all $i$, and $\{a_1y_1, \ldots, a_ny_n\}$ is a basis for $H$.  Notice that the index of $H$ in $\Z^n$ is clearly $a_1a_2 \ldots a_n$, so for this product to be the square-free integer $k$ (given that $a_i$ divides $a_{i+1}$ for every $i$), the only possibility is that $a_1=\ldots=a_{n-1}=1$ and $a_n=k$.  Thus, there is a basis $\{y_1, \ldots, y_n\}$ for $\Z^n$ such that $\{y_1, \ldots, y_{n-1},ky_n\}$ is a basis for $H$, so taking $\sigma$ to be the automorphism of $\Z^n$ that takes $y_n$ to $e_1$, $y_1$ to $e_n$, and $y_i$ to $e_i$ for $2 \le i \le n-1$, where $e_1, \ldots, e_n$ is the standard basis for $\Z^n$, establishes the desired result.
\end{proof}

We are now ready to give our characterisation.

\begin{thm}\label{main}
Let $\Gamma=\Cay(\Z^n,S)$ be non-empty and locally finite, with $n>1$.  Then $\Gamma$ is a (D)CI-graph if and only if:
\begin{itemize}
\item $\Gamma$ (or its underlying graph) has a finite, square-free number of components; and
\item $\Aut(H)= \Aut(H)_{\Z^n}\cdot\Stab_{\Aut(H)}(S)$,
\end{itemize}
where $H = \langle S \rangle$, $\Stab_{\Aut(H)}(S)$ is the group of all automorphisms of $H$ that fix $S$ setwise, and $\Aut(H)_{\Z^n}$ is the group of all automorphisms of $H$ that can be extended to automorphisms of $\Z^n$.
\end{thm}

\begin{proof}
($\Rightarrow$)  We assume that $\Gamma$ is a (D)CI-graph.  By Proposition~\ref{non-CI-on-Zn}, $\Gamma$ must have a finite, square-free number of components.

Take any automorphism $\beta$ of $H$.  Then $\Cay(\Z^n,\beta(S)) \cong \Gamma$, so since $\Gamma$ is a (D)CI-graph, there must be some $\gamma \in \Aut(\Z^n)$ such that $\gamma(S)=\beta(S)$.  So $\gamma^{-1}\beta \vert_{H} \in \Aut(H)$ and fixes $S$.  Hence $\gamma^{-1}\beta\vert_{H} \in \Stab_{\Aut(H)}(S)$.  Also since $\beta \in \Aut(H),$ $H=\langle S \rangle$, and $\beta(S)=\gamma(S)$, we have $\gamma(H)=\beta(H)=H$, so $\gamma\vert_{H} \in \Aut(H)$.  Hence $\gamma\vert_H \in \Aut(H)_{\Z^n}$.  Therefore $\beta=(\gamma\vert_H)(\gamma^{-1}\beta\vert_H) \in \Aut(H)_{\Z^n} \cdot \Stab_{\Aut(H)}(S)$.  This shows that $\Aut(H) \le \Aut(H)_{\Z^n}  \cdot\Stab_{\Aut(H)}(S)$; since both of the groups in the product are subgroups of $\Aut(H)$, the other inclusion is immediate.

($\Leftarrow$)  Suppose that $\Gamma \cong \Gamma'=\Cay(\Z^n,S')$.  Let $H =\langle S \rangle$ and $H'=\langle S' \rangle$.  Let $k$ be the number of connected components of $\Gamma$ (and therefore of $\Gamma'$), so by assumption $k$ is finite and square-free.  Then $|H:\Z^n|=|H':\Z^n|=k$.  Since $k$ is finite, the rank of $H$ (and of $H'$) is also $n$.

By Corollary~\ref{conj-H}, we can conjugate both $H$ and $H'$ to $k\Z\times \Z^{n-1}$ using an element of $\Aut(\Z^n)$, so $H$ and $H'$ are conjugate to each other in $\Aut(\Z^n)$.  Thus, replacing $S'$ by a conjugate if necessary, we may assume without loss of generality that $H'=H$.

Now since $H'=H\cong \Z^n$ and we have $\Cay(H,S)\cong \Cay(H',S')=\Cay(H,S')$ is connected, Corollary~\ref{Zn-is-CI_f} tells us that this is a (D)CI-graph, so there is some $\tau \in \Aut(H)$ such that $\tau(S)=S'$.  By assumption, $\tau=\tau_1\tau_2$ where $\tau_1 \in \Aut(H)_{\Z^n}$ and $\tau_2 \in \Stab_{\Aut(H)}(S)$.  Now, since $\tau_2$ fixes $S$, we have $\tau_1(S)=\tau\tau_2^{-1}(S)=\tau(S)=S'$.  By definition of $\Aut(H)_{\Z^n}$, there is some $\sigma' \in \Aut(\Z^n)$ such that $\sigma'\vert_H=\tau_1$, so since $S \subseteq H$, we have $\sigma'(S)=\tau_1(S)=S'$.  This has shown that there is an automorphism of $\Z^n$ taking $S$ to $S'$, so $\Gamma$ is a (D)CI-graph.
\end{proof}

To demonstrate the importance of the rather odd-looking condition in our characterisation, that $\Aut(H)=\Stab_{\Aut(H)}(S) \cdot \Aut(H)_{\Z^n}$, we conclude with some examples in which this condition is not satisfied (so the graphs are not CI, despite having a finite and in many cases square-free number of connected components) and some examples in which it is satisfied (so the graphs are CI).

\begin{eg} The graph $\Cay(\Z^n,S)$ where $e_1, \ldots, e_n$ is the standard basis for $\Z^n$, $m>1$, and $S=\{\pm me_1, \pm e_2, \ldots, \pm e_n\}$, is not a CI-graph.
\end{eg}
\begin{proof}
Let $H=\langle S \rangle$.  Let $S'=\{\pm(me_1+e_2), \pm e_2, \ldots, \pm e_n\}$.  Clearly $S'$ and $S$ are both bases for $H$, so there is some automorphism $\sigma$ of $H$ that takes $S$ to $S'$.  By multiplying by an element of $\Stab_{\Aut(H)}(S)$ if necessary, if our condition were to hold, we would be able to find such a $\sigma$ that would extend to an automorphism of all of $\Z^n$. But since every entry of $me_1\in S$ is a multiple of $m$, and nothing in $S'$ has this property, there is no automorphism of $\Z^n$ that takes $me_1$ into $S'$, so in particular, $\sigma$ cannot extend to an automorphism of $\Z^n$.
\end{proof}

\begin{eg}  The graph $\Cay(\Z^2,S)$ where $S=\{\pm(2,0), \pm (0,1), \pm (2,1)\}$ satisfies the condition, so is a CI graph.
\end{eg}
\begin{proof}
Let $H=\langle S\rangle = 2\Z \times \Z$, so the graph has 2 components, which is a finite, square-free number.  Thus we only need to check the second condition of Theorem~\ref{main} to see that this is a CI graph.

Let $$\sigma=\begin{pmatrix} 1/2&0\\ 0 & 1\end{pmatrix}.$$  Then 
$\sigma$ is an isomorphism from $H$ to $\Z^2$, and $$\sigma(S)=\{\pm(1,0), \pm(0,1), \pm(1,1)\}.$$  Now, the stabiliser of $\sigma(S)$ in $\Aut(\Z^2)$ contains $$\Bigl\langle \begin{pmatrix}0 & -1 \\ 1 & -1\end{pmatrix} \Bigr\rangle,$$
which has order 3 mod $\pm I$.

Also, $$\Aut(H)_{\Z^2} = \left\{\phi\vert_H: \phi=\begin{pmatrix} a& b \\ c& d \end{pmatrix}, b\text{ is even and } ad-bc=\pm 1\right\},$$ where $\phi\vert_H$ denotes the restriction of $\phi$ to its action on $H$, so $$\sigma(\Aut(H)_{\Z^2})\sigma^{-1}=\left\{\begin{pmatrix} a& b \\ c& d \end{pmatrix}: c\text{ is even and } ad-bc=\pm 1\right\}.$$ This has index 3 in $\sigma \Aut(H)\sigma^{-1}=\mathsf{GL}(2,\Z)$, because if we consider the natural homomorphism onto $\mathsf{GL}(2,\Z_2)$, the image of this subgroup  consists of 2 of the 6 elements of $\mathsf{GL}(2,\Z_2)$, so has index 3.

Since the order of $\Stab_{\Aut(\Z^2)}(\sigma(S))$ has order at least 3, which is the index of $\sigma\Aut(H)_{\Z^2}\sigma^{-1}$ in $\mathsf{GL}(2,\Z)$, and the intersection of the subgroup of order 3 with $\sigma\Aut(H)_{\Z^2}\sigma^{-1}$ is trivial, we must have 
$$\sigma(\Aut(H)_{\Z^2})\sigma^{-1}\cdot\Stab_{\Aut(\Z^2)}(\sigma(S)) =\mathsf{GL}(2,\Z),$$ so conjugating by $\sigma^{-1}$ gives
$$ \Aut(H)_{\Z^2}\cdot\Stab_{\Aut(H)}(S)=\Aut(H),$$ satisfying the condition, as claimed.
\end{proof}

In fact, it turns out that the second condition of our characterisation will never be satisfied if the number of connected components in the Cayley (di)graph is sufficiently large relative to $n$.

\begin{cor}
For every $n >1$, there exists some natural number $k_n$ such that if $\Gamma=\Cay(\Z^n,S)$ is nonempty and locally finite with at least $k_n$ connected components, then $\Gamma$ is not a (D)CI-graph.
\end{cor}

\begin{proof}
It is well-known (cf. the stronger result \cite[Theorem 4.3]{Plat-Rap} that implies this) that given $n$, there exists $k_n$ such that every finite subgroup of $\Aut(\Z^n)$ has order less than $k_n$.  

Let $\Gamma=\Cay(\Z^n,S)$ be a nonempty locally-finite (di)graph.  If the number of connected components of $\Gamma$ is infinite or square-free, then Proposition~\ref{non-CI-on-Zn} tells us that $\Gamma$ is not a (D)CI-graph.  So we can assume that $\Gamma$ has $k$ connected components, where $k$ is square-free and $H \cong \Z^n$, and that $k \ge k_n$. 
By Corollary~\ref{conj-H}, we can conjugate $S$ by an element of $\Aut(\Z^n)$ if necessary, to ensure that $H =k\Z \times \Z^{n-1}$.

Let $$\sigma=\begin{pmatrix}1/k&0&\ldots &0\\0&1&\ldots&0\\ \vdots&&\vdots\\0&0&\ldots&1\end{pmatrix}.$$ Then $\sigma$ is an isomorphism from
$H$ to  $\Z^n$. 
Since every finite subgroup of $\Aut(\Z^n)$ has order less than $k_n$, in particular $$|\Stab_{\Aut(\Z^n)}(\sigma(S))| < k_n.$$
Now $\Aut(H)_{\Z^n}$ is $$\left\{\phi\vert_H:\phi=\begin{pmatrix} b_{11}&\ldots& b_{1n} \\ \vdots& &\vdots\\b_{n1}&\ldots & b_{nn} \end{pmatrix}, \rm{det}(\phi)= \pm 1 \text{ and }b_{12},\ldots, b_{1n} \equiv 0 \ \mod{k} \right\},$$ so $\sigma(\Aut(H)_{\Z^n})\sigma^{-1}$ is $$\left\{\begin{pmatrix} b_{11}&\ldots& b_{1n} \\ \vdots& &\vdots\\b_{n1}&\ldots & b_{nn} \end{pmatrix}: \text{ determinant is} \pm 1 \text{ and }b_{21},\ldots, b_{n1} \equiv 0 \ \mod{k} \right\}.$$ 
We claim that this has index greater than $k$ in $\mathsf{GL}(n,\Z)=\sigma\Aut(H)\sigma^{-1}$.
If we consider the natural homomorphism onto $\mathsf{GL}(n,\Z_k)$, the subgroup of $\mathsf{GL}(n,\Z_k)$ consisting of matrices of the form 
$$\begin{pmatrix}1&0&0&\ldots&0\\ x&1&0&\ldots&0\\0&0&1&\ldots&0\\\vdots&&&\vdots&\vdots\\0&0&0&\ldots&1\end{pmatrix}$$  
where $x \in \Z_k$, has order $k$ and intersects the image of $\sigma(\Aut(H)_{\Z^n})\sigma^{-1}$ (which is a subgroup) in only $I$, so the index of the image of $\sigma(\Aut(H)_{\Z^n})\sigma^{-1}$ in $\mathsf{GL}(n,\Z_k)$ under this homomorphism must be at least $k$.  Hence the index of $\sigma(\Aut(H)_{\Z^n})\sigma^{-1}$ in $\mathsf{GL}(n,\Z)$ is at least $k$.

Now, if $\Gamma$ were to be (D)CI, then by Theorem~\ref{main}, we would have $\Aut(H)=\Aut(H)_{\Z^n} \cdot \Stab_{\Aut(H)}(S)$, so conjugating by $\sigma$, 
$$\mathsf{GL}(n,\Z)=\sigma\Aut(H)_{\Z^n}\sigma^{-1}\cdot \Stab_{\Aut(\Z^n)}(\sigma(S)).$$
  In particular, it would certainly need to be true that the index of $\sigma\Aut(H)_{\Z^n}\sigma^{-1}$ in $\mathsf{GL}(n,\Z)$ is no bigger than the order of $\Stab_{\Aut(\Z^n)}(\sigma(S))$.  But we have 
shown that this index is at least $k$, and that the order is less than $k_n \le k$,
 a contradiction that shows that $\Gamma$ cannot be (D)CI.
\end{proof}

\section{Acknowledgments}
The author is deeply indebted to Laci Babai for many long conversations about this problem that led to the terminology introduced in this paper. Ideas from these conversations also formed the basis of many of the results and proofs presented here.

\end{document}